\documentclass{amsart}
\usepackage[breaklinks,colorlinks,linkcolor=blue,citecolor=blue,
urlcolor=blue]{hyperref}
\usepackage{amssymb,enumerate }
\usepackage{graphicx}
\graphicspath{ {images/} }
\newtheorem{theorem}{Theorem}[section]
\newtheorem{lemma}[theorem]{Lemma}
\newtheorem{corollary}[theorem]{Corollary}
\newtheorem{problem}[theorem]{Problem}

\newtheorem{remark}[theorem]{Remark}
\newtheorem{definition}[theorem]{Definition}
\newtheorem{example}[theorem]{Example}
\numberwithin{equation}{section}
\numberwithin{equation}{section}

\begin{document}
	
		\title{Star versions of Lindel\"{o}f spaces}
	
	\author{ Sumit Singh }

	\address{Department of Mathematics, Dyal Singh College,  University of Delhi, Lodhi Road, New Delhi-110003, India.}
\email{ sumitkumar405@gmail.com}

	
	
	\dedicatory{}
	
	

	\maketitle
	
	\begin{abstract} A space $ X $ is said to be  set star-Lindel\"{o}f (resp., set strongly star-Lindel\"{o}f) if for each nonempty subset $ A $ of $ X $ and each collection $ \mathcal{U} $ of open sets in $ X $ such that  $ \overline{A} \subseteq \bigcup \mathcal{U} $, there is a countable subset  $ \mathcal{V}$  of $ \mathcal{U} $ (resp., countable subset $ F $ of $ \overline{A} $) such that $ A  \subseteq  {\rm St}( \bigcup \mathcal{V}, \mathcal{U})$ (resp., $ A  \subseteq  {\rm St}( F, \mathcal{U})$). The classes of set star-Lindel\"{o}f spaces and set strongly star-Lindel\"{o}f spaces lie between the class of Lindel\"{o}f spaces and the class of star-Lindel\"{o}f spaces. In this paper, we investigate the relationship among set star-Lindel\"{o}f spaces, set strongly star-Lindel\"{o}f spaces, and other related spaces by providing some suitable examples and study the topological properties of set star-Lindel\"{o}f and set strongly star-Lindel\"{o}f  spaces.
	\end{abstract}

\section{Introduction and Preliminaries}
A cardinal  function $ sL $ defined from a set of topological spaces to a cardinal number $ sL(X) $. Arhangel'skii \cite{AAV} defined a cardinal number $ sL(X) $ of $ X $: the minimal infinite cardinality $ \tau $  such that for every subset $ A\subset X $ and every open cover $ \mathcal{U} $ of $ \overline{A} $, there is a subfamily $ \mathcal{V} \subset \mathcal{U} $ such that $ |\mathcal{V}| \leq \tau $ and $ A \subseteq \overline{\bigcup \mathcal{V}}$.  If $ sL(X)=\omega $, then the space $ X $ is called \textit{sLindel\"{o}f space}. Following this idea, Ko\v{c}inac and Konca  \cite{KK} introduced and studied the new types of selective covering properties called set-covering properties. A space $ X $ is said to have the set-Menger \cite{KK} property if for each nonempty subset $ A $ of $ X $ and each sequence $ (\mathcal{U}_n: n \in \mathbb{N}) $ of collections of  open sets in $ X $ such that for each $ n \in \mathbb{N} $, $ \overline{A} \subseteq \bigcup \mathcal{U}_n $, there is a sequence $ (\mathcal{V}_n: n \in \mathbb{N}) $ such that for each $ n \in \mathbb{N} $, $ \mathcal{V}_n $ is a finite subset of $ \mathcal{U}_n $ and  $ A \subseteq \bigcup_{n \in \mathbb{N}}  \bigcup \mathcal{V}_n $. The author \cite{SS} noticed that the set-Menger property is nothing but another view of Menger covering property.  Recently,  the author \cite{SSCSS}  defined and studied set starcompact and set strongly starcompact spaces (also see \cite{SKK}).

In this paper, we apply the above study of Ko\v{c}inac, Konca, and Singh to defined a new subclass of star-Lindel\"{o}f spaces called set star-Lindel\"{o}f and set strongly star-Lindel\"{o}f spaces. With the help of some suitable example, we investigate the relationship among set star-Lindel\"{o}f, set strongly star-Lindel\"{o}f, and other related spaces. A space  having a dense Lindel\"{o}f subspace is star-Lindel\"{o}f (see \cite{SYN}). We show that this result is not true if we replace star-Lindel\"{o}f space with a set star-Lindel\"{o}f space. 

If $ A  $ is a subset of a space $ X $ and  $\mathcal{U}$ is a collection of subsets of $ X $,   then ${\rm St}(A,\mathcal{U}$)=$ \bigcup \{U \in \mathcal{U} : U \cap A \not= \emptyset \} $. We usually write ${\rm St}(x,\mathcal{U})={\rm St}( \{x\},\mathcal{U}$).

Throughout the paper, by ``a space" we mean ``a topological space", $ \mathbb{N} $, $ \mathbb{R} $ and $ \mathbb{Q} $ denotes the set of natural numbers, set of real numbers, and set of rational numbers, respectively, the cardinality of a set is denoted by $ |A| $. Let $ \omega $ denote the first infinite cardinal, $ \omega_1 $ the first uncountable cardinal, $ \mathfrak{c} $ the cardinality of the set of all real numbers. An open cover $ \mathcal{U} $ of a subset $ A \subset X$ means elements of $ \mathcal{U} $ open in $ X $ such that $A \subseteq \bigcup \mathcal{U}=\bigcup \{U:U \in \mathcal{U} \} $.

We first recall the classical notions of spaces that are used in this paper. 

\begin{definition} \cite{DEK}
	A space X is said to be
	\begin{enumerate}
		\item starcompact if for each open cover $ \mathcal{U} $ of  X, there is a finite subset  $ \mathcal{V}$  of $ \mathcal{U} $ such that $ X  =  {\rm St}( \bigcup \mathcal{V}, \mathcal{U}) $.
		\item strongly starcompact if for each open cover $ \mathcal{U} $ of  X, there is a finite subset  $F$  of $ X $ such that $ X  =  {\rm St}(F, \mathcal{U}) $.
	\end{enumerate}
\end{definition}

\begin{definition} \cite{SSCSS,SKK}
	A space X is said to be
	\begin{enumerate}
		\item set starcompact if for each nonempty subset A of X and each collection $ \mathcal{U} $ of open sets in X such that  $ \overline{A} \subseteq \bigcup \mathcal{U} $, there is a finite subset  $ \mathcal{V}$  of $ \mathcal{U} $ such that $ A  \subseteq  {\rm St}( \bigcup \mathcal{V}, \mathcal{U})$.
		\item set strongly starcompact if for each nonempty subset A of X and each collection $ \mathcal{U} $ of open sets in X such that  $ \overline{A} \subseteq \bigcup \mathcal{U} $, there is a finite subset  $F$  of $ \overline{A} $ such that $ A  \subseteq  {\rm St}(F, \mathcal{U})$.
	\end{enumerate}
\end{definition}

\begin{definition}
	A space X is said to be
	\begin{enumerate}
		\item star-Lindel\"{o}f \cite{DEK} if for each open cover $ \mathcal{U} $ of  $ X $, there is a countable subset  $ \mathcal{V}$  of $ \mathcal{U} $ such that $ X  =  {\rm St}( \bigcup \mathcal{V}, \mathcal{U})$.
		\item strongly star-Lindel\"{o}f \cite{DEK} if for each open cover $ \mathcal{U} $ of  $ X $, there is a countable subset  $ F$  of $ X $ such that $ X  = {\rm St}( F, \mathcal{U})$.
	\end{enumerate}
\end{definition}
Note that the star-Lindel\"{o}f spaces have a different name such as 1-star-Lindel\"{o}f and 1$ \frac{1}{2} $-star-Lindel\"{o}f in different papers (see \cite{DEK,MMV}) and the strongly star-Lindel\"{o}f space is also called star countable in \cite{MMV,XS}. It is clear that, every strongly star-Lindel\"{o}f space is star-Lindel\"{o}f.

Recall that a collection $ \mathcal{A} \subseteq P(\omega) $ is said to be almost disjoint  if each set $  A \in \mathcal{A} $ is infinite and the sets $ A \bigcap B $ are finite for all distinct elements $ A, B \in \mathcal{A} $. For an almost disjoint family $ \mathcal{A} $, put $ \psi(\mathcal{A})= \mathcal{A} \bigcup \omega $ and topologize $ \psi(\mathcal{A}) $  as follows:  for each element $ A \in \mathcal{A} $ and each finite set $ F \subset \omega $, $ \{A\} \bigcup (A \setminus F) $ is a basic open neighborhood of $ A $ and the natural numbers are isolated. The spaces of this type are called Isbell-Mr\'{o}wka $ \psi $-spaces \cite{BM,MS}  or $ \psi (\mathcal{A}) $ space. For other terms and symbols we follow \cite{ER}.

The following result was proved in \cite{SKK}. 

\begin{theorem} \cite{SKK} \label{1.4}	Every countably compact space is set strongly starcompact.
\end{theorem}

\section{set star-Lindel\"{o}f and related spaces}
In this section, we give some examples showing the relationship among set star-Lindel\"{o}f spaces, set strongly star-Lindel\"{o}f spaces, and other related spaces. First we define our main definition.

\begin{definition}
	A space X is said to be
	\begin{enumerate}
		\item set star-Lindel\"{o}f if for each nonempty subset A of X and each collection $ \mathcal{U} $ of open sets in X such that  $ \overline{A} \subseteq \bigcup \mathcal{U} $, there is a countable subset  $ \mathcal{V}$  of $ \mathcal{U} $ such that $ A \subseteq  {\rm St}( \bigcup \mathcal{V}, \mathcal{U}) $.
		\item set strongly star-Lindel\"{o}f if for each nonempty subset A of X and each collection $ \mathcal{U} $ of open sets in X such that  $ \overline{A} \subseteq \bigcup \mathcal{U} $, there is a countable subset  $F$  of $ \overline{A} $ such that $ A \subseteq {\rm St}(F, \mathcal{U})$.
	\end{enumerate}
\end{definition}

The following result shows that the class of set strongly star-Lindel\"{o}f spaces is big enough.

\begin{theorem} \label{2.2}  For a space X, the following statements are hold:
	\begin{enumerate}
		\item If X is a  Lindel\"{o}f space, then X is set strongly star-Lindel\"{o}f.
		
		\item If X is a countably compact space, then X is set strongly star-Lindel\"{o}f. 
	\end{enumerate}
\end{theorem}
\begin{proof} (1). Let $ A \subseteq X $ be a nonempty set and $ \mathcal{U} $ be a collection of open sets such that $ \overline{A} \subseteq \bigcup \mathcal{U} $. Since space $ X $ is Lindel\"{o}f, closed subset $ \overline{A} $ of $ X $ is also Lindel\"{o}f. Thus there exists a countable subset $ \mathcal{V} $ of $ \mathcal{U} $ such that $ A \subseteq \overline{A} \subseteq \bigcup \mathcal{V} $. Choose $ x_V \in \overline{A} \cap V $ for each $ V \in \mathcal{V} $. Let $ F= \{x_V: V \in \mathcal{V} \} $. Then $ F $ is a countable subset of $ \overline{A} $ and $ A \subseteq \bigcup \mathcal{V} \subseteq {\rm St}(F, \mathcal{U}) $. Therefore $ X $ is set strongly star-Lindel\"{o}f.
	
	(2). Since  every set strongly starcompact space is set strongly star-Lindel\"{o}f, by Theorem \ref{1.4}, $ X $ is set strongly star-Lindel\"{o}f.  \end{proof}

We have the following diagram from the definitions and Theorem \ref{2.2}. However, the following examples shows that the converse of these implications are not  true.

\[
\hskip2cm { set \ strongly \ starcompact} \rightarrow { set \ starcompact}
\]
\[
\hskip3.0cm \downarrow  \hskip3.5cm \downarrow
\]
\[ { Lindel{o}f} \rightarrow { set\ strongly \ star-Lindel{o}f}  \rightarrow { set \ star-Lindel{o}f}
\]
\[
\hskip3.0cm \downarrow  \hskip3.5cm \downarrow
\]
\[
\hskip1.9cm { strongly \ star-Lindel{o}f}  \rightarrow  { star-Lindel{o}f}
\]\newline

\begin{example} \label{2.1}
	There exists Tychonoff set strongly star-Lindel\"{o}f (hence, set star-Lindel\"{o}f) space which is not set starcompact (hence, not set strongly starcompact).
\end{example}
\begin{proof}Let $ X=\omega $ be the discrete space. Then X is set strongly star-Lindel\"{o}f  but not set starcompact  space. \end{proof}

\begin{example}
	Let $ X=[0, \omega_1) $. Then $ X $ is countably compact space but not Lindel\"{o}f and not separable. Thus $ X $ is set strongly star-Lindel\"{o}f space which is not Lindel\"{o}f.
\end{example}

The following example shows that the converse of Theorem \ref{2.2}(2)  is not true.

\begin{example}
	Let Y be a discrete space with cardinality $ \mathfrak{c} $. Let $ X= Y\cup \{y^* \}$, where $ y^* \notin  Y $. Define topology on $ X $, each $ y \in Y $ is an isolated point and a set $ U $ containing $ y^* $ is open if and only if $ X \setminus U $ is countable. Then $ X $ is Lindel\"{o}f, hence set strongly star-Lindel\"{o}f. 
	
	Now to show $ X $ is not countably compact. Let $ y^* \in U $ such that $ X \setminus U= \{y_\alpha: \alpha < \omega \} $, where $ y_\alpha \in Y $. Then $ U $ is open in X and $  \mathcal{U}= \{U\} \bigcup \{y_\alpha: \alpha < \omega \} $ is a countable cover of X, which does not have a finite subcover. Thus X is not countably compact.
\end{example}

\begin{example} \label{2.44}
	There exists a Tychonoff strongly star-Lindel\"{o}f space which is not a set strongly star-Lindel\"{o}f.
\end{example}
\begin{proof} Let $ X=\psi(\mathcal{A})= \mathcal{A} \cup \omega $ be the Isbell-Mr\'{o}wka space with $ |A|=\omega_1 $. Then $ X $ is strongly star-Lindel\"{o}f Tychonoff pseudocompact space, since $ X $ is separable. Now we prove that $ X $ is not set strongly star-Lindel\"{o}f. Let $ A= \mathcal{A} = \{a_\alpha: \alpha< \omega_1 \}$. Then $ A $ is closed subset of $ X $. For each $ \alpha< \omega_1 $, let $ U_\alpha= \{a_\alpha \} \cup (a_\alpha) $. Let $ \mathcal{U}= \{U_\alpha: \alpha< \omega_1 \} $. Then $ \mathcal{U} $ is an  open cover of $ \overline{A} $. It is enough to show that there exists a point $ a_\beta  \in A $ such that 
	\begin{center}
		$ a_\beta \notin {\rm St}(F, \mathcal{U}) $,
	\end{center}
	for any countable subset $ F$  of $ \overline{A} $. Let $ F $ be any countable subset of $ \overline{A} $. Then there exists $ \alpha'< \omega_1 $ such that $ U_\alpha \cap  F =\emptyset$, for each $ \alpha > \alpha' $. Pick $ \beta > \alpha' $, then $ U_\beta \cap F =\emptyset$. Since $ U_\beta $ is the only element of $ \mathcal{U} $ containing the point $ a_\beta $. Thus $ a_\beta \notin {\rm St}(F, \mathcal{U}) $, which shows that $ X $ is not set strongly star-Lindel\"{o}f. \end{proof}

The following lemma was proved by Song \cite{SYN}.

\begin{lemma}  $ [ $\cite{SYN}, Lemma 2.2$ ] $ \label{2.8}
	A space $ X $ having a dense Lindel\"{o}f subspace is star-Lindel\"{o}f.
\end{lemma}

The following example shows that the Lemma \ref{2.8} does not hold if we replace star-Lindel\"{o}f space by a set star-Lindel\"{o}f space.

\begin{example} \label{2.9}
	There exists a Tychonoff space X having a dense Lindel\"{o}f subspace such that X is not set star-Lindel\"{o}f.
\end{example}
\begin{proof}
	Let $ D(\mathfrak{c})=\{d_\alpha: \alpha<\mathfrak{c} \} $ be a discrete space of cardinality $ \mathfrak{c} $ and let $ Y=D(\mathfrak{c})\cup \{d^* \} $ be one-point compactification of $ D(\mathfrak{c}) $. Let 
	
\begin{center}
	
$ X=(Y \times [0, \omega)) \cup (D(\mathfrak{c}) \times \{\omega \}) $ \end{center} be the subspace of the product space $ Y \times [0,\omega] $. Then $ Y \times [0, \omega) $ is a dense Lindel\"{o}f subspace of $ X $ and by Lemma \ref{2.8}, $ X $ is star-Lindel\"{o}f.
	
	Now we show that $ X $ is not set star-Lindel\"{o}f. Let $ A= D(\mathfrak{c}) \times \{\omega \} $. Then $ A $ is the closed subset of $ X $. For each $ \alpha< \mathfrak{c} $, let $ U_\alpha= \{d_\alpha \} \times [0,\omega] $. Then \begin{center}
	
	$ U_\alpha \cap U_{\alpha'}=\emptyset $ for $ \alpha \not= \alpha' $. \end{center}  Let 
	
\begin{center} $ \mathcal{U}= \{U_\alpha:\alpha< \mathfrak{c}  \} $. \end{center}  Then $ \mathcal{U} $ is an open cover of $ \overline{A} $. It is enough to show that there exists a point $ \langle d_\beta, \omega \rangle \in A $ such that $  \langle d_\beta, \omega \rangle \notin {\rm St}(\bigcup \mathcal{V}, \mathcal{U}) $ for any countable subset $ \mathcal{V} $ of $ \mathcal{U} $. Let $ \mathcal{V} $ be any countable subset of $ \mathcal{U} $. Then there exists $ \alpha'< \mathfrak{c} $ such that $ U_\alpha \notin \mathcal{V} $ for each $ \alpha> \alpha' $. Pick $ \beta > \alpha' $. Then $ U_\beta \bigcap (\bigcup \mathcal{V})=\emptyset $, but $ U_\beta $ is the only element of $ \mathcal{U} $ containing $  \langle d_\beta, \omega \rangle $. Thus $  \langle d_\beta, \omega \rangle \notin {\rm St}(\bigcup \mathcal{V}, \mathcal{U}) $. Therefore $ X $ is not set star-Lindel\"{o}f. \end{proof}

\begin{example} 
	There exists a $ T_1 $ set star-Lindel\"{o}f space $ X $ that is not set strongly star-Lindel\"{o}f.
\end{example}
\begin{proof} Let $ X= A \cup B $, where $ A= [0,\mathfrak{c}) $  and $ B =\{b_n:n \in \omega \} $ and for each $ n \in \omega $,  $ b_n \notin A $. Topologize $ X $ as follows: for each $ \alpha \in A $ and each finite subset $ F \subset B $, $ \{\alpha \} \cup (B \setminus F) $ is a basic open neighborhood of $ \alpha $ and for each $ n \in \omega $, $ b_n $ is isolated. Then $ X $ is a $ T_1 $-space. Let $ C $ be any nonempty subset of $ X $ and $ \mathcal{U} $ be an open cover of $ \overline{C} $.
	
	First we show that $ X $ is set star-Lindel\"{o}f space. For this we have three possible cases:
	
	Case (i):  If $ C \subset A $. Then for each $ \alpha \in C $, there exists $ U_{\alpha} \in \mathcal{U} $ such that $ \alpha \in U_{\alpha} $. Then for each $ \alpha \in C $, we can find a finite set $ F_{\alpha} $ such that $ \{\alpha \} \cup (B \setminus F_{\alpha}) \subseteq U_{\alpha} $. It is clear that  for each $ \alpha \not= \alpha' $, $ U_{\alpha} \cap U_{\alpha'} \not= \emptyset $. Let $ \mathcal{V}'= \{U_{\alpha} \} $. Then $ C \subset {\rm St}(\bigcup \mathcal{V}', \mathcal{U}) $. 
	
	Case (ii): If $ C \subset B $. Since $ B $ is countable, thus B is set star-Lindel\"{o}f. Hence we have a countable subset $ \mathcal{V}'' $ of $ \mathcal{U} $ such that $ C \subset {\rm St}(\bigcup \mathcal{V}'', \mathcal{U}) $. 
	
	Case (iii): If $ C=C_1 \cup C_2 $ such that where $ C_1 \subset A $ and $ C_2 \subset B $. Choose $ \mathcal{V}' $ and $ \mathcal{V}'' $ from Case (i) and Case (ii), respectively such that $ C_1 \subset {\rm St}(\bigcup \mathcal{V}', \mathcal{U}) $  and $ C_2 \subset {\rm St}(\bigcup \mathcal{V}'', \mathcal{U}) $. Let $ \mathcal{V}= \mathcal{V} \bigcup \mathcal{V}''$. Then $ \mathcal{V} $ is a countable subset of $ \mathcal{U} $ and $ C \subset {\rm St}(\bigcup \mathcal{V}', \mathcal{U}) $. 
	
	Thus $ X $ is a set star-Lindel\"{o}f space.

	Now we prove that $ X $ is not set strongly star-Lindel\"{o}f space.  Since $ A=[0, \mathfrak{c}) $ is a closed subset of $ X $. For each $ \alpha< \mathfrak{c} $, let $ U_\alpha= \{a_\alpha \} \cup (a_\alpha) $. Then $ \mathcal{U}= \{U_\alpha: \alpha< \mathfrak{c} \} $ is an  open cover of $ A $. It is enough to show that there exists a point $ \beta  \in A $ such that 
	\begin{center}
		$ \beta \notin St(F, \mathcal{U}) $
	\end{center}
	for any countable subset $ F $   of $ A $. Let $ F $ be a countable subset of $ A $. Then there exists $ \alpha'< \mathfrak{c} $ such that $ U_\alpha \cap F =\emptyset$. Pick $ \beta > \alpha' $, then $ U_\beta \cap F =\emptyset$. Since $ U_\beta $ is the only element of $ \mathcal{U} $ containing the point $ \beta $. Thus $ \beta \notin St(F, \mathcal{U}) $, which shows that $ X $ is not set strongly star-Lindel\"{o}f. \end{proof}

\begin{remark}
	(1) In \cite{SSCSS}, Singh gave an example of a Tychonoff set starcompact space $ X $ that is not set strongly starcompact. 
	
	(2) It is known that there are star-Lindel\"{o}f spaces that are not strongly star-Lindel\"{o}f (see [\cite{DEK}, Example 3.2.3.2] and [\cite{DEK}, Example 3.3.1]).
\end{remark}

Now we give some conditions under which star-Lindel\"{o}fness coincide with set star-Lindel\"{o}fness and strongly star-Lindel\"{o}fness coincide with set strongly   star-Lindel\"{o}fness.

Recall that a space $ X $ is paraLindel\"{o}f if every open cover $ \mathcal{U} $ of $ X $ has a locally countable open refinement.

Song and Xuan \cite{SYKX} proved the following result.
\begin{theorem} $ [ $\cite{SYKX}, Theorem 2.24$ ] $ \label{2.10}
	Every regular  paraLindel\"{o}f star-Lindel\"{o}f spaces are Lindel\"{o}f.
\end{theorem}

We have the following theorem from Theorem \ref{2.10} and the diagram.

\begin{theorem} 
	If X is a regular paraLindel\"{o}f space, then the following statements are equivalent:
	\begin{enumerate}
		\item X is Lindel\"{o}f;
		\item X is set strongly star-Lindel\"{o}f;
		\item X is set star-Lindel\"{o}f;
		\item X is strongly star-Lindel\"{o}f;
		\item X is star-Lindel\"{o}f.
	\end{enumerate}	
\end{theorem}

A space is said to be metaLindel\"{o}f if every open cover of it has a point-countable open refinement.

Xuan and Shi \cite{XS}  proved the following result.	
	\begin{theorem} $ [ $\cite{XS}, Proposition 3.12$ ] $ \label{2.12}
		Every strongly star-Lindel\"{o}f metaLindel\"{o}f spaces are Lindel\"{o}f.
	\end{theorem}
We have the following theorem from Theorem \ref{2.12} and the diagram.

\begin{theorem} \label{2.6}
	If X is a metaLindel\"{o}f space, then the following statements are equivalent:
\begin{enumerate}
		\item X is Lindel\"{o}f;
	\item X is set strongly star-Lindel\"{o}f;
	\item X is strongly star-Lindel\"{o}f.
\end{enumerate}	
\end{theorem}

	\section{Properties of set star-Lindel\"{o}f spaces and set strongly star-Lindel\"{o}f spaces}
	In this section, we study the topological properties of set star-Lindel\"{o}f and set strongly star-Lindel\"{o}f spaces.

\begin{theorem} \label{3.3}
	If X is a  set star-Lindel\"{o}f space, then every open and closed subset of X is set star-Lindel\"{o}f.
\end{theorem}
\begin{proof} Let $ X $ be a set star-Lindel\"{o}f space and $ A \subseteq X $ be an open and closed set. Let $ B $ be any subset of $ A $ and $ \mathcal{U} $ be a collection of open sets in $ (A, \tau_A) $ such that $ Cl_A (B) \subseteq \bigcup \mathcal{U} $. Since $ A $ is open, then $ \mathcal{U} $ is a collection of open sets in $ X $. Since $ A $ is closed, $ Cl_A(B)=Cl_X (B) $. Applying the set star-Lindel\"{o}fness property of $ X $, there exists a countable subset $ \mathcal{V} $  of $ \mathcal{U} $ such that $ B \subseteq St (\bigcup \mathcal{V},\mathcal{U})  $. Hence $ A $ is a set star-Lindel\"{o}f.\end{proof}

Similarly, we can prove the following.
\begin{theorem}
	If X is a set strongly star-Lindel\"{o}f space, then every open and closed subset of X is set strongly star-Lindel\"{o}f.
\end{theorem}

Consider the Alexandorff duplicate $ A(X)=X \times \{0,1\} $ of a space $ X $. The basic neighborhood of a point $ \langle x,0 \rangle \in X \times \{0\} $ is of the form $ (U \times \{0\}) \bigcup (U \times \{1\} \setminus \{\langle x,1 \rangle\}) $, where $ U $ is a neighborhood of $ x $ in $ X $ and each point  $ \langle x,1 \rangle \in X \times \{1\} $ is a isolated point.

\begin{theorem} \label{3.7}
	If X is a $ T_1 $-space and A(X) is a set star-Lindel\"{o}f space. Then $ e(X)< \omega_1 $.
\end{theorem}
\begin{proof}Suppose that $ e(X) \geq \omega_1 $. Then there exists a discrete closed subset $ B $ of $ X $ such that $ |B| \geq \omega_1 $. Hence $ B \times \{1\} $ is an open and closed subset of $ A(X) $ and every point of $ B \times \{1\} $ is an isolated point. Thus $ A(X) $ is not set star-Lindel\"{o}f, by Theorem \ref{3.3}, every open and closed subset of a set star-Lindel\"{o}f space is set star-Lindel\"{o}f and $ B \times \{1\} $ is not set star-Lindel\"{o}f. \end{proof}

\begin{corollary}
	If X is a $ T_1 $-space and A(X) is a set strongly star-Lindel\"{o}f space. Then $ e(X)< \omega_1 $.
\end{corollary}

\begin{theorem} 
	Let X be a space such that the Alexandorff duplicate $ A(X) $ of X is set star-Lindel\"{o}f (resp., set strongly star-Lindel\"{o}f). Then $ X $ is a set star-Lindel\"{o}f (resp., set strongly star-Lindel\"{o}f) space.
\end{theorem}

\begin{proof} Because the proof of the two cases are quite similar, we only prove the case of set star-Lindel\"{o}f.
	
	 Let $ B $ be any nonempty subset of $ X $ and $ \mathcal{U} $ be an open cover of  $ \overline{B}$. Let $ C =B \times \{ 0\} $ and
	
\begin{center}$ A(\mathcal{U})= \{U\times \{0,1 \} : U \in \mathcal{U} \}$. \end{center} Then $ A(\mathcal{U}) $ is an open cover of $ \overline{C} $. Since $ A(X) $ is set star-Lindel\"{o}f, there is a countable subset  $ A(\mathcal{V}) $  of $ A(\mathcal{U}) $ such that $ C \subseteq  {\rm St}(\bigcup A(\mathcal{V}), A(\mathcal{U})) $. Let 
	
\begin{center} $ \mathcal{V}= \{U \in \mathcal{U}: U \times \{0,1 \} \in A(\mathcal{V}) \} $. \end{center} Then  $ \mathcal{V} $ is a countable subset of $ \mathcal{U} $. Now we have to show that
	
	\begin{center}$ B \subseteq  {\rm St}(\bigcup \mathcal{V}, \mathcal{U}) $. \end{center}  Let $ x \in B$. Then $ \langle x,0 \rangle \in {\rm St}(\bigcup A(\mathcal{V}), A(\mathcal{U})) $. Choose $ U \times \{0,1 \} \in A(\mathcal{U})$ such that $ \langle x,0 \rangle \in U \times \{0,1 \} $ and $ U \times \{0,1 \} \cap (\bigcup A(\mathcal{V})) \not= \emptyset $, which implies $ U \cap (\bigcup \mathcal{V})\not= \emptyset $ and $ x \in U $. Therefore $ x \in {\rm St}(\bigcup \mathcal{V}, \mathcal{U}) $, which shows that $ X $ is set star-Lindel\"{o}f space. \end{proof}

On the images of set star-Lindel\"{o}f spaces, we have the following result.

\begin{theorem} \label{3.10}
	A continuous image of set star-Lindel\"{o}f space is set star-Lindel\"{o}f.
\end{theorem}
\begin{proof}Let $ X $ be a set star-Lindel\"{o}f space and $ f:X \rightarrow Y $ is a continuous mapping from $ X$ onto $ Y $. Let $ B $ be any subset of $ Y $ and   $ \mathcal{V} $ be an open cover of $ \overline{B} $. Let $ A =f^{-1} (B) $. Since $ f $ is continuous,  $ \mathcal{U}= \{f^{-1}(V): V \in \mathcal{V}\}$ is the collection of open sets in $ X $ with $ \overline{A}= \overline{f^{-1} (B)}\subseteq f^{-1}(\overline{B}) \subseteq f^{-1} (\bigcup \mathcal{V})=\bigcup \mathcal{U}$. As $ X $ is set star-Lindel\"{o}f, there exists a countable subset $ \mathcal{U}' $  of $ \mathcal{U} $ such that
	
\begin{center}$ A \subseteq  {\rm St}(\bigcup \mathcal{U}', \mathcal{U}) $. \end{center} Let $ \mathcal{V}'= \{V:f^{-1}(V) \in \mathcal{U}'\}$. Then $ \mathcal{V}' $ is a countable subset of $ \mathcal{V} $ and $ B=f(A) \subseteq f( {\rm St}(\bigcup \mathcal{U}', \mathcal{U})) \subseteq {\rm St}(\bigcup f(\{f^{-1}(V): V \in \mathcal{V}'\}), \mathcal{V})= {\rm St}(\bigcup \mathcal{V}', \mathcal{V})$. Thus $ Y $ is set star-Lindel\"{o}f space.\end{proof}

Similarly, we can prove the following.
\begin{theorem} 
	A continuous image of a set strongly star-Lindel\"{o}f space is set strongly star-Lindel\"{o}f.
\end{theorem}

Next, we turn to consider preimages of set strongly star-Lindel\"{o}f and set star-Lindel\"{o}f spaces. We need new concepts called nearly set strongly star-Lindel\"{o}f and nearly set star-Lindel\"{o}f spaces.  A space $ X $ is said to be nearly set strongly star-Lindel\"{o}f (resp., nearly set star-Lindel\"{o}f) in $ X $  if for each subset $ Y $ of $ X $  and each open cover $ \mathcal{U} $ of  $ X $, there is a countable subset $ F $  of $ X $ (resp., a countable subset $ \mathcal{V} $  of $ \mathcal{U} $) such that $ Y \subseteq  {\rm St}(F, \mathcal{U}) $ (resp., $ Y \subseteq  {\rm St}(\bigcup \mathcal{V}, \mathcal{U}) $). 

\begin{theorem}
	Let $ f : X \rightarrow Y $ be an open, closed, and finite-to-one continuous mapping from a space X onto a set strongly star-Lindel\"{o}f space Y. Then X is  nearly  set strongly star-Lindel\"{o}f.
\end{theorem}
\begin{proof} Let $ A \subseteq X$ be any nonempty set and $ \mathcal{U} $ be an open cover of  $ X  $. Then $ B=f(A) $ is a subset of $ Y $. Let $ y \in \overline{B} $. Then $ f^{-1} \{y\} $ is a finite subset of $ X $, thus there is a finite subset $ \mathcal{U}_y $ of $ \mathcal{U} $  such that $ f^{-1} \{y\} \subseteq \bigcup  \mathcal{U}_y $ and $ U \bigcap f^{-1} \{y\}\not= \emptyset $ for each $ U \in \mathcal{U}_y $. Since $ f $ is closed, there exists an open neighborhood $ V_y $ of $ y $ in $ Y $ such that $ f^{-1}(V_y)\subseteq \bigcup \{U:U \in \mathcal{U}_y \} $. Since $ f $ is open, we can assume that
	
	\begin{center}$ V_y \subseteq \bigcap \{f(U):U \in \mathcal{U}_y \} $. \end{center} Then $ \mathcal{V}=\{V_y: y \in \overline{B} \} $ is an open cover of $ \overline{B} $. Since $ Y $ is set strongly   star-Lindel\"{o}f,  there exists a countable subset $ F $ of $ \overline{B} $ such that  $ B \subseteq {\rm St}(F,\mathcal{V})$. Since $ f $ is finite-to-one, then $ f^{-1}(F) $
	is a countable subset of $ X $. We have to show that 
	
\begin{center}$ A \subseteq   {\rm St}( f^{-1}(F), \mathcal{U}) $. \end{center} Let $ x \in A $. Then there exists $ y \in B $ such that $ f(x) \in V_y $ and $ V_y \bigcap F \not= \emptyset $. Since

\begin{center} $x \in  f^{-1}  (V_y) \subseteq \bigcup \{U:U \in \mathcal{U}_y \} $, \end{center} we can choose $ U \in \mathcal{U}_y $ with $ x \in U $. Then $ V_y \subseteq f(U) $. Thus $ U \bigcap f^{-1}  (F) \not= \emptyset  $. Hence $ x \in {\rm St}(f^{-1}  (F) ,\mathcal{U}) $.  Therefore $ X $ is nearly  set strongly star-Lindel\"{o}f.	\end{proof}

\begin{theorem} \label{3.16}
	If $ f : X \rightarrow Y $ is an  open and perfect continuous mapping and Y is a set star-Lindel\"{o}f space, then X is  nearly  set star-Lindel\"{o}f.
\end{theorem}
\begin{proof} Let $ A \subseteq X$ be any nonempty set and $ \mathcal{U} $ be an open cover of  $ X  $. Then $ B=f(A) $ is a subset of $ Y $. Let $ y \in \overline{B} $. Then $ f^{-1} \{y\} $ is a compact subset of $ X $, thus there is a finite subset $ \mathcal{U}_y $ of $ \mathcal{U} $  such that $ f^{-1} \{y\} \subseteq \bigcup  \mathcal{U}_y $. Let $ U_y= \bigcup  \mathcal{U}_y $. Then $ V_y=Y \setminus f(X \setminus U_y) $ is a neighborhood of $ y $, since $ f $ is closed. Then $ \mathcal{V}=\{V_y: y \in \overline{B} \} $ is an open cover of $ \overline{B} $. Since $ Y $ is set star-Lindel\"{o}f,  there exists a countable subset $ \mathcal{V}' $   of $ \mathcal{V} $ such that 	\begin{center} $ B \subseteq  {\rm St}(\bigcup \mathcal{V}',\mathcal{V}) $. \end{center}  Without loss of generality, we may assume that
	$ \mathcal{V}'=\{V_{y_i}: i \in N' \subseteq \mathbb{N} \} $. Let $ \mathcal{W}= \bigcup_{i \in N'} \mathcal{U}_{y_i} $. Since $ f^{-1}  (V_{y_i}) \subseteq \bigcup \{U:U \in \mathcal{U}_{y_i} \} $ for each $ i \in N' $. Then $ \mathcal{W} $ is a countable subset of $ \mathcal{U}$ and
	
		\begin{center}$ f^{-1} (\bigcup\mathcal{V}') = \bigcup \mathcal{W} $. \end{center} Next, we show that

	\begin{center}$ A \subseteq {\rm St}(\bigcup \mathcal{W},\mathcal{U}) $. \end{center} Let $ x \in A$. Then there exists a  $ y \in B $ such that

\begin{center}$f(x) \in V_y $ and $ V_y \bigcap (\bigcup\mathcal{V}') \not= \emptyset$. \end{center} Since

\begin{center} $x \in  f^{-1}  (V_y) \subseteq \bigcup \{U:U \in \mathcal{U}_y \} $, \end{center} we can choose $ U \in \mathcal{U}_y $ with $ x \in U $. Then $ V_y \subseteq f(U) $. Thus $ U \bigcap f^{-1}  (\bigcup \mathcal{V}') \not= \emptyset $. Hence $ x \in {\rm St}( f^{-1}  (\bigcup \mathcal{V}'),\mathcal{U}) $. Therefore $ x \in {\rm St}( \bigcup \mathcal{W},\mathcal{U}) $, which shows that $ A \subseteq  {\rm St}(\bigcup \mathcal{W},\mathcal{U}) $. Thus $ X $ is nearly set star-Lindel\"{o}f.	\end{proof}

It is known that the product of star-Lindel\"{o}f space and  compact space is a star-Lindel\"{o}f (see \cite{DEK}).
\begin{problem} \label{3.19}
	Does the product of set star-Lindel\"{o}f space and a compact space is set star-Lindel\"{o}f?
\end{problem}

However, the product of two set strongly star-Lindel\"{o}f spaces need not be set strongly star-Lindel\"{o}f. The following well-known example shows that the product of two countably compact (hence, set strongly star-Lindel\"{o}f) spaces need not be set star-Lindel\"{o}f.  Here we give the roughly proof for the sake of completeness.

\begin{example} \label{3.20}
	There exist two countably compact  spaces X and Y such that $ X \times Y $ is not set star-Lindel\"{o}f (hence, not set strongly star-Lindel\"{o}f).
\end{example}
\begin{proof}
Let $ D(\mathfrak{c}) $ be a discrete space of the cardinality $ \mathfrak{c} $. We can define $ X= \bigcup_{\alpha< \omega_1} E_\alpha $ and $ Y= \bigcup_{\alpha< \omega_1} F_\alpha $, where $ E_\alpha $ and $ F_\alpha $ are the subsets of $ \beta (D(\mathfrak{c})) $ which are defined inductively to satisfy the following three conditions:

\begin{enumerate}[(1)]
	\item $ E_\alpha \bigcap F_\beta=D(\mathfrak{c})  \ if \ \alpha \not= \beta$;
	\item $ |E_\alpha| \leq \mathfrak{c} $ and $  |F_\alpha| \leq \mathfrak{c} $;
	\item every infinite subset of $ E_\alpha $ (resp., $ F_\alpha $) has an accumulation point in $ E_{\alpha+1} $ (resp, $ F_{\alpha+1} $).
\end{enumerate}
Those sets $ E_\alpha $ and $ F_\alpha $ are well-defined since every infinite closed set in $ \beta (D(\mathfrak{c})) $ has the cardinality $ 2^\mathfrak{c} $ (see \cite{W}). Then, $ X \times Y $ is not set star-Lindel\"{o}f, and the diagonal $ \{\langle d, d \rangle: d \in D(\mathfrak{c}) \} $ is a discrete open and closed subset of $ X \times Y $ with the cardinality $ \mathfrak{c} $. Thus $ X \times Y $ is not set star-Lindel\"{o}f, since the open and closed subset of set star-Lindel\"{o}f space is set star-Lindel\"{o}f and the diagonal $ \{\langle d, d \rangle: d \in D(\mathfrak{c}) \} $ is not set star-Lindel\"{o}f.
\end{proof}

\begin{remark}
	Example \ref{3.20}, shows that the product of set star-Lindel\"{o}f space and countably compact space need not be set star-Lindel\"{o}f.
\end{remark}
  van Douwen-Reed-Roscoe-Tree [\cite{DEK}, Example 3.3.3] gave an example of a countably compact $ X $ (hence, set star-Lindel\"{o}f) and a Lindel\"{o}f space $ Y $ such that  $ X \times Y$ is not strongly star-Lindel\"{o}f. Now we use this example to show that $ X \times Y$ is not set star-Lindel\"{o}f.

\begin{example} \label{3.22}
	There exists a countably compact (hence, set strongly star-Lindel\"{o}f) space $ X $ and a Lindel\"{o}f space Y such that $ X \times Y$ is not set star-Lindel\"{o}f. 
\end{example}
\begin{proof}Let $ X=[0,\omega_1) $ withe usual order topology. Let $ Y=[0, \omega_1] $ with the following topology. Each point $ \alpha < \omega_1$ is isolated and a set $ U $ containing $ \omega_1 $ is open if and only if $ Y \setminus U $ is countable. Then, $ X $ is countably compact and $ Y $ is Lindel\"{o}f. It is enough to show that $ X \times Y $ is not star-Lindel\"{o}f, since every set star-Lindel\"{o}f space is star-Lindel\"{o}f.

For each $ \alpha< \omega_1 $,  $ U_\alpha= X \times \{\alpha\} $ is open in $ X \times Y$. For each $ \beta< \omega_1 $,  $ V_\beta= [0,\beta] \times (0,\omega_1] $ is open in $ X \times Y$. Let $ \mathcal{U}= \{U_\alpha: \alpha< \omega_1 \} \cup \{V_\beta: \beta< \omega_1 \} $. Then $ \mathcal{U} $ is an open cover of $ X \times Y$. Let $ \mathcal{V} $ be any countable subset of $ \mathcal{U} $. Since $ \mathcal{V} $ is countable, there exists $ \alpha'< \omega_1 $ such that $ U_\alpha \notin \mathcal{V} $ for each $ \alpha> \alpha' $.  Also,  there exists $ \alpha''< \omega_1 $ such that $ V_\beta \notin \mathcal{V} $ for each $ \beta> \alpha'' $. Let  $ \beta =sup \{\alpha', \alpha''\} $. Then $ U_\beta \bigcap (\bigcup \mathcal{V}) =\emptyset$ and $ U_\beta$ is the only element containing $ \langle \beta , \alpha \rangle $. Thus $ \langle \beta , \alpha \rangle \notin {\rm St}(\bigcup \mathcal{V}, \mathcal{U}) $, which shows that $ X $ is not star-Lindel\"{o}f. \end{proof}

van Douwen-Reed-Roscoe-Tree [\cite{DEK}, Example 3.3.6] gave an example of Hausdorff regular Lindel\"{o}f spaces $ X $ and $ Y $ such that  $ X \times Y$ is  star-Lindel\"{o}f. Now we use this example and show that the product of two Lindel\"{o}f spaces is not set star-Lindel\"{o}f.

\begin{example}
	There exists a Hausdorff regular Lindel\"{o}f spaces $ X $ and $ Y $ such that $ X \times Y$ is not set star-Lindel\"{o}f.
\end{example}
\begin{proof} Let $ X=\mathbb{R}\setminus \mathbb{Q} $ have the induced metric topology. Let $ Y=\mathbb{R} $ with each point of $ \mathbb{R}\setminus \mathbb{Q} $ is isolated and points of $ \mathbb{Q} $ having metric neighborhoods. Hence both spaces $ X $ and $ Y $ are Hausdorff regular Lindel\"{o}f spaces and first countable too, so $ X \times Y$ Hausdorff regular and first countable. Now we show that  $ X \times Y$ is not set star-Lindel\"{o}f. Let $ A= \{(x,x)\in X \times Y: x\in X \} $. Then $ A $ is uncountable closed and discrete set (see [\cite{DEK}, Example 3.3.6]). For $ (x,x) \in A $, $ U_x=X \times \{x\} $ is open subset of $ X \times Y $. Then $ \mathcal{U}= \{U_x:(x,x) \in \overline{A} \} $ is an open cover of $ \overline{A} $. Let $ \mathcal{V} $ be any countable subset of $ \mathcal{U} $. Then there exists $ (a,a) \in A $ such that $ (a,a) \notin \bigcup \mathcal{V} $ and thus $ (\bigcup \mathcal{V}) \bigcap U_a=\emptyset $. But $ U_a $ is the only element of $ \mathcal{U} $ containing $ (a,a) $. Thus $ (a,a) \notin {\rm St}(\bigcup \mathcal{V},\mathcal{U}) $, which completes the proof. \end{proof}

\end{document}